\documentclass[12pt]{amsart}
\usepackage{amscd, amssymb, amsmath}
\usepackage{color}

\newtheorem{thmenum}{Theorem}[section]
\newtheorem{prpenum}[thmenum]{Proposition}

\newtheorem{corenum}[thmenum]{Corollary}

\theoremstyle{definition}
\newtheorem{dfnenum}[thmenum]{Definition}
\newtheorem{exmenum}[thmenum]{Example}

\theoremstyle{remark}
\newtheorem{rmkenum}[thmenum]{Remark}

\def\Tm{{\mathcal T}}
\def\Xm{{\mathcal X}}
\def\Ym{{\mathcal Y}}
\def\Zm{{\mathcal Z}}
\def\Integer{{\mathbb Z}}
\def\add{\operatorname{add}}
\def\mod{\operatorname{mod}}
\def\Hom{\operatorname{Hom}}
\def\Ext{\operatorname{Ext}}
\def\pro{\operatorname{proj}}
\def\iso{\cong}

\def\La{\Lambda}

\def\td{\mathrm{\mathop{tri.dim}}}

\def\gd{\mathrm{\mathop{gl.dim}}}
\def\id{\mathrm{\mathop{inj.dim}}}

\def\ya#1{\overset{#1}{\longrightarrow}}
\def\blank{\operatorname{-}}

\title{Relative derived dimensions for cotilting modules}
\author{Michio Yoshiwaki}
\address{
\begin{flushleft}
        \hspace{0.3cm}  Department of Mathematics \\
         \hspace{0.3cm}  Faculty of Science\\
         \hspace{0.3cm}  Shizuoka University \\
         \hspace{0.3cm}  836 Ohya, Suruga-ku, Shizuoka, 422-8529, JAPAN\\
\end{flushleft}
\begin{flushleft}
	\hspace{0.3cm} Osaka City University Advanced Mathematical Institute,\\
	\hspace{0.3cm} 3-3-138 Sugimoto, Sumiyoshi-ku, Osaka, 558-8585, JAPAN\\
\end{flushleft}
}
\email{yoshiwaki.michio@shizuoka.ac.jp}
\date{\today}
\thanks{The author was partially supported by JST (Japan Science and Technology Agency) CREST Mathematics Grant  (15656429)}
\begin{document}
\maketitle
\begin{abstract}
For a Noetherian ring $R$ and
a cotilting $R$-module $T$  of injective dimension at least $1$, 
we prove that the derived dimension of $R$ 
with respect to the category $\Xm_T$ 
is precisely the injective dimension of $T$ by applying Auslander-Buchweitz theory
and Ghost Lemma.
In particular, when $R$ is a commutative Noetherian local ring with a canonical module $\omega_R$
and $\dim R\ge1$, the derived dimension of R with respect to the category of maximal
Cohen-Macaulay modules is precisely $\dim R$.
\end{abstract}

\section{Introduction} \label{section:Introduction}

The notion of dimension of a triangulated category was introduced by Rouquier \cite{R2}
based on work of Bondal and Van den Bergh \cite{BV} on Brown representability.
A relative version of this notion was introduced in \cite{ABIM}, which
counts how many extensions are needed to build the triangulated category out of a given subcategory.
The aim of this paper is to give an explicit value of the relative dimension when
the subcategory is associated with a cotilting modules.

In this paper, 
we denote by $R$ a Noetherian ring. 
All $R$-modules are finitely generated right $R$-modules. 
We denote by $\mod R$ the abelian category of $R$-modules and 
by $\mathsf{D}^{\mathrm{b}}(\mod R)$ the derived category of $\mod R$.

Then our main result is the following, which completes a main result Theorem 5.3 in \cite{AAITY}.
\begin{thmenum} \label{CD}
Let $R$ be a Noetherian ring  
and $T$ a cotilting $R$-module with $ \id\ T\geq 1$. 
Then we have an equality 
$$\Xm_T \mbox{--} \td\ \mathsf{D}^{\mathrm{b}}(\mod R) = \id\  T.$$
\end{thmenum}

The inequality $\leq$ was shown in \cite[Theorem 5.3]{AAITY}.  
In this paper, we will prove the converse inequality by applying Auslander-Buchweitz theory and Ghost Lemma.

We apply Theorem 1.1 to the following settings.
For a commutative Noetherian local ring $R$ with a canonical module $\omega_R$,
we denote by $\mathsf{CM} R$ the category of maximal Cohen-Macaulay modules.
We call an $R$-algebra $\Lambda$ an \emph{$R$-order} if $\Lambda\in\mathsf{CM} R$.
We denote by $\mathsf{CM} \Lambda$ the category of maximal Cohen-Macaulay
$\Lambda$-modules (i.e. $\Lambda$-modules $X$ satisfying $X\in\mathsf{CM} R$).
As a special case of Theorem~\ref{CD}, we obtain the following results, which completes
the inequalities (1.2.1) and (4.2.1) in \cite{AAITY}.

\begin{corenum} \label{cor:comm}
Let $R$ be a commutative Noetherian local ring with a canonical module $\omega_R$ and $\dim R\ge1$.
Then
\begin{enumerate}
\item[$(1)$] We have an equality
$$(\mathsf{CM} R)\mbox{--} \td\ \mathsf{D}^{\mathrm{b}}(\mod R)=\dim R.$$
\item[$(2)$] More generally, for an $R$-order $\Lambda$, we have an equality
$$(\mathsf{CM}\La)\mbox{--} \td\ \mathsf{D}^{\mathrm{b}}(\mod \La)=\dim R.$$
\end{enumerate}
\end{corenum}

\begin{proof}
Since $\omega_R$ (respectively, $\omega_\La:=\Hom_R(\La,\omega_R)$)
is a cotilting module with injective dimension $\dim R$, the assertion follows from Theorem 1.1.
\end{proof}

\begin{rmkenum}{\cite[Remark 5.4]{AAITY}}
If $\id\ \omega_\La=\dim R=0$, 
then the equality in Corollary~\ref{cor:comm} (2) does not necessarily hold in general.
Namely, 
let $\La$ be a finite dimensional non-semisimple self-injective algebra over a field.
Then the right $\La$-module $\La$ is a cotilting module with $\id \ \La =0$ and $\Xm_{\La} = \mod \La$. 
However, $\langle \mod \La \rangle$ is different from $\mathsf{D}^{\mathrm{b}}(\mod \La)$. 
\end{rmkenum}

\section{Preliminaries}
In this section, we will introduce several concepts.
\begin{dfnenum}{(Aihara-Araya-Iyama-Takahashi-Y \cite{AAITY})}
\ 

Let $\Tm$ be a triangulated category with shift $[1]$ and 
$\Xm,\Ym$ full subcategories of $\Tm$.\\
$(1)$ The full subcategory $\Xm\ast\Ym$ of $\Tm$ is defined as follows:
$$\Xm\ast\Ym:=\{ M\in \Tm\ |\ \exists \text{ a triangle}: X\to M \to Y \to X[1] 
\ \text{with}\  X\in\Xm,Y\in\Ym \}.$$
Note that $(\Xm\ast \Ym)\ast \Zm=\Xm\ast(\Ym\ast \Zm) $ holds by the octahedral axiom.\\
$(2)$ Set $\langle\Xm\rangle:=\add\{X[i]\ |\ X\in\Xm, i\in\Integer \}$. 
And, for any positive integer $n$,
$$
\langle\Xm\rangle_n
:=\add(\underbrace{\langle\Xm\rangle*\langle\Xm\rangle*\cdots*\langle\Xm\rangle}_{n}).
$$
Clearly, $\langle\Xm\rangle_n$ is closed under shifts. \\
$(3)$ The {\it dimension} of $\Tm$ with respect to a subcategory $\Xm$ is defined as follows:
$$
\Xm \mbox{--} \td\ \Tm:=\inf\{n\geq 0\ |\ \Tm=\langle \Xm \rangle_{n+1} \}.
$$
\end{dfnenum}

When $\Xm = \add M$ for some object $M\in\Tm$, one can recover the dimension of triangulated category in the sense of Rouquier \cite{R2}.\\

The relative (derived) dimensions realize the several invariants for rings.
For instance, we have the following fact, which was proved by Krause and Kussin.  

\begin{exmenum}{(Krause-Kussin \cite[Lemma 2.4 and 2.5]{KK})} \label{E2}
$$(\pro R) \mbox{--} \td\ \mathsf{D}^{\mathrm{b}}(\mod R) =\gd\ R,$$
where $\pro R$ is the subcategory of $\mod R$ consisting of projective modules and $\gd\ R$ is the global dimension of $R$.
This is a special case of our main result.
\end{exmenum}

For an $R$-module $T$, 
we define the full subcategory ${}^{\perp} T$ of $\mod R$ as follows:
$${}^{\perp} T:=\{ X\in \operatorname{mod} R\ |\  \operatorname{Ext}_R^i(X, T)=0\ \text{for any}\  i >0 \}.$$ 
Then we will introduce the concept of a cotilting module. 
\begin{dfnenum}
An $R$-module $T$ is called {\it cotilting} if it satisfies the following three conditions:
\begin{enumerate}
\item[$(1)$] The injective dimension $\id\ T$ of $T$ is finite. 
\item[$(2)$] $T\in {}^{\perp} T$.
\item[$(3)$] For any $X\in {}^{\perp} T$, there exists a short exact sequence
$$0\to X\to T'\to X'\to0$$
with $T'\in\operatorname{add} T,\ X'\in {}^{\perp} T$.
\end{enumerate}
\end{dfnenum}

For a cotilting module $T$, we will write $\Xm_T$ instead of ${}^{\perp} T$. 
Moreover, we set $$\Ym_T:= (\Xm_T)^{\perp}= \{ Y\in \mod R \ |\ \Ext^i_R(X,Y)=0 \text{ for any } i>0 \text{ and any } X \in \Xm_T  \}. $$ 
Then by Auslander-Buchweitz approximation theory \cite{AB}, we have the following fact.

\begin{prpenum} \label{prp:cotorsion}
Let $T$ be a cotilting $R$-module with injective dimension $d$ and $M\in \mod R$. Then
\begin{enumerate}
\item[$(1)$] there exists an exact sequence $0 \to Y \to X \to M \to 0$ with $X\in \Xm_T$ and $Y \in \Ym_T$. 
\item[$(2)$] $M$ belongs to $\Ym_T$ if and only if there exists an exact sequence $$0\to T_d \to \cdots \to T_1 \to T_0 \to M \to 0$$ with $T_i \in \add T$.
\end{enumerate}
\end{prpenum}


A typical example of a cotilting module is the following.
\begin{exmenum}
$(1)$ The canonical module over a commutative Noetherian local ring $R$ is a cotilting $R$-module.\\
$(2)$ For a finite dimensional algebra $R$ over a field $k$ and a tilting $R^{op}$-module $T$ in the sense of \cite{M}, the $k$-dual of $T$ is a cotilting $R$-module.
\end{exmenum}

\section{Proof of our result} \label{section:proof}
We need the following to prove Theorem~\ref{CD}, which is called the {\em Ghost Lemma}. 

\begin{prpenum}{\cite[Lemma 4.11]{R2}} \label{lem:ghost}
Let $H_1, H_2, \cdots, H_{n+1}$ be cohomological functors on a triangulated category $\Tm$ 
and $f_i : H_i \to H_{i+1}$ morphisms between them. 
Let $\Xm_i$ be subcategories of $\Tm$ such that $f_i$ vanishes on $\Xm_i$ and $\Xm_i = \langle \Xm_i \rangle$.
Then the composite $f_n \circ \cdots \circ f_1$ vanishes on $\add (\Xm_1 \ast \cdots \ast \Xm_n)$. 
\end{prpenum}

Now we will prove Theorem~\ref{CD}.

\begin{proof}[Proof of Theorem~\ref{CD}]
In the rest, we show the converse inequality. 

Set $\id\ T=: d \geq 1$.  
Then we have $\mathsf{D}^{\mathrm{b}}(\mod R)(M,T[d])\iso \Ext^d_R (M, T) \not = 0$ for some $M \in \mod R$.
We will identify $\mathsf{D}^{\mathrm{b}}(\mod R)(-,-[i])$ with $\Ext^i_R (-, -) $ for any integer $i$ under the natural isomorphism. 
By Proposition~\ref{prp:cotorsion}, 
we have an exact sequence  
$$
[\xi_M : 0\to T_d \ya{\phi_d} \cdots \ya{\phi_2} T_1 \ya{\phi_1} T_0 \ya{\phi_0} M \to 0] \in \Ext^d_R (M, T),
$$
where $T_0 \in \Xm_T$ and $T_i \in \add T = \Xm_T \cap \mathcal{Y}_T$ for all $i=1,\cdots, d$.
Note that $\xi_M \not = 0$ since $\Ext^d_R (M, T)\not= 0$. 
Put $K_i := \operatorname{Im} \phi_i$ for each $i=0,\cdots,d$. 
Then for any $i=1, \cdots,d$, $K_i \in \mathcal{Y}_T $ 
and we have a short exact sequence
$$
\xi_i : 0 \to K_i \to T_{i-1} \to K_{i-1} \to 0.
$$
Regarding $\xi_M$ (respectively, $\xi_i$) as a morphism from $K_0=M$ to $K_d[d]=T_d[d]$
(respectively, from $K_{i-1}[i-1]$ to $K_{i}[i]$) in $\mathsf{D}^{\mathrm{b}}(\mod R)$, we have an equality
$$\xi_M=\xi_{d}\circ\cdots\circ\xi_1.$$

For any $i=1,\cdots d$, 
the morphism $\xi_i$ induces a morphism 
$$ f_i: \mathsf{D}^{\mathrm{b}}(\mod R) (\blank , K_{i-1}[i-1]) \to \mathsf{D}^{\mathrm{b}}(\mod R) (\blank, K_i[i]). $$
Clearly $\mathsf{D}^{\mathrm{b}}(\mod R) (\Xm_T,K_i[j])=0 $ holds if $(i,j)$ belongs to $\{0,\cdots, d-1 \}\times \Integer_{<0} $ or 
$\{1,\cdots, d\}\times \Integer_{>0}$.  
Therefore $\mathsf{D}^{\mathrm{b}}(\mod R) (\Xm_T, \xi_i [ j ])=0 $ holds for any $ i=1,\cdots,d$ and $ j\in \Integer$.
Namely, $f_i$ vanishes on $\langle \Xm_T \rangle$. 
By Proposition~\ref{lem:ghost}, the composite $f_{d} \circ \cdots \circ f_1$ vanishes on $\langle \Xm_T \rangle_d$.  
But the composite $f_{d} \circ \cdots \circ f_1$  sends the identity morphism of $K_{0} = M$ to $\xi_M=\xi_{d} \circ \cdots \circ \xi_1 \not=0$ 
and hence $M\not\in \langle \Xm_T \rangle_d$. 
Namely, we have the inequality
$$
\Xm_T \mbox{--} \td\ \mathsf{D}^{\mathrm{b}}(\mod R) \geq \id\ T,
$$
as desired.
\end{proof}


\begin{thebibliography}{9}
   \bibitem{AAITY}
  T.\ Aihara, T.\ Araya, O.\ Iyama, R.\ Takahashi and M.\ Yoshiwaki, 
  {\em Dimensions of triangulated categories with respect to subcategories}, 
   J.\ Algebra \textbf{399}  (2014),\ 205--219. 
  \bibitem{AB}
  M.\ Auslander and R.-O.\ Buchweitz, 
  {\em The homological theory of maximal Cohen-Macaulay approximations}, (French summary) 
  Colloque en l'honneur de Pierre Samuel (Orsay, 1987). 
  M\' em. Soc. Math. France (N.S.) No. 38 (1989),\ 5--37.
  \bibitem{ABIM}
  L.L.\ Avramov, R.-O.\ Buchweitz, S.B.\ Iyengar and C.\ Miller,  
  {\em Homology of perfect complexes},
  Adv. Math. \textbf{223} (2010), no. 5,\ 1731--1781 and corrigendum, Adv. Math. \textbf{225} (2010), no. 6,\ 3576--3578.
  \bibitem{BV}
  A.\ Bondal and M.\ Van den Bergh, 
  {\em Generators and representability of functors in commutative and noncommutative geometry},
  Mosc.\ Math.\ J. \textbf{3} (2003), no. 1,\ 1--36, 258. 
  \bibitem{KK}
  H.\ Krause and D.\ Kussin, 
  {\em Rouquier's theorem on representation dimension}, 
  Trends in representation theory of algebras and related topics,\ 95-103,
  Contemp.\ Math., \textbf{406}, Amer. Math. Soc., Providence, RI, 2006. 
  \bibitem{M}
  Y.\ Miyashita, 
  {\em Tilting modules of finite projective dimension}, 
  Math. Z. \textbf{193} (1986), no. 1,\ 113--146. 
  \bibitem{R2} 
  R.\ Rouquier,
  {\em Dimensions of triangulated categories}, 
  J.\ K-theory \textbf{1} (2008), no. 2,\ 193--256 and errata,\ 257--258.
  \end{thebibliography}
\end{document}